\documentclass[british]{amsart}

\pdfoutput=1

\usepackage{babel}
\usepackage[T1]{fontenc}
\usepackage[utf8]{inputenc}
\usepackage[varg]{pxfonts}

\usepackage{needspace}
\usepackage{setspace}
\usepackage{aliascnt}
\usepackage{varioref}
\labelformat{equation}{\textup{(#1)}}
\labelformat{enumi}{\textup{(#1)}}

\usepackage{amssymb,mathrsfs}

\usepackage{ifpdf}
\ifpdf
\usepackage{microtype}
\fi

\usepackage[pdfusetitle, unicode=true, bookmarks=true, pdfborder={0 0 0}, colorlinks=false, breaklinks=true, linktoc=all]{hyperref}

\hypersetup{breaklinks=false, colorlinks=true, linkcolor=blue, citecolor=blue, urlcolor=blue, linktoc=page} 

\title{Involutive \texorpdfstring{$A_\infty$--}{A-infinity }algebras and dihedral cohomology}

\author{Christopher Braun}
\address{Centre for Mathematical Science\\
City University London\\
Northampton Square\\
London EC1V 0HB\\UK}
\email{Christopher.Braun.1@city.ac.uk}


\theoremstyle{plain}
\newtheorem{theorem}{Theorem}[section]
\newcommand{\newautoreftheorem}[2]{
\newaliascnt{#1}{theorem}\newtheorem{#1}[#1]{#2}\aliascntresetthe{#1}%
\expandafter\def\csname #1autorefname\endcsname{#2}}

\newautoreftheorem{proposition}{Proposition}
\newautoreftheorem{lemma}{Lemma}
\newautoreftheorem{corollary}{Corollary}
\newtheorem*{theorem*}{Theorem}

\theoremstyle{definition}
\newautoreftheorem{definition}{Definition}
\newautoreftheorem{remark}{Remark}
\newautoreftheorem{example}{Example}
\newautoreftheorem{examples}{Examples}
\newautoreftheorem{notation}{Notation}
\newautoreftheorem{convention}{Convention}
\newtheorem*{notation*}{Notation}

\numberwithin{equation}{section} 
\numberwithin{figure}{section}   

\newcommand{\co}{\colon\thinspace}
\newcommand{\dual}{\mathbf{D}}
\newcommand{\mob}{\mathrm{M}}
\newcommand{\cfree}[2]{\widehat{\mathcal{F}}_{#1}[#2]}
\newcommand{\com}{\mathcal{C}om}
\newcommand{\ass}{\mathcal{A}ss}
\newcommand{\lie}{\mathcal{L}ie}
\newcommand{\CC}{\mathrm{CC}}
\newcommand{\HCC}{\mathrm{HC}}
\newcommand{\CD}{\mathrm{CD}}
\newcommand{\HCD}{\mathrm{HD}}

\newcommand{\HCE}{\mathrm{HCE}}
\newcommand{\hoch}{\mathrm{CH}}
\newcommand{\Hhoch}{\mathrm{HH}}

\newcommand{\Hharr}{\mathrm{HHarr}}
\newcommand{\MC}{\mathrm{MC}}
\newcommand{\MCmoduli}{\mathscr{MC}}
\newcommand{\Def}{\mathrm{Def}}
\newcommand{\Der}{\mathrm{Der}}
\newcommand{\cycl}{\mathrm{cycl}}
\newcommand{\id}{\mathrm{id}}
\newcommand{\degree}[1]{\lvert #1 \rvert}
\newcommand{\reduction}[1]{\overline{#1}}
\DeclareMathOperator{\ad}{ad}
\DeclareMathOperator{\Hom}{Hom}
\DeclareMathOperator{\Aut}{Aut}
\DeclareMathOperator{\IHom}{\underline{Hom}}
\newcommand{\dgvect}{\mathbf{dgVect}_k}
\newcommand{\fdgvect}{\mathscr{F}\dgvect}
\newcommand{\dgop}{\mathbf{dgOp}}

\begin{document}
\def\sectionautorefname{Section}

\begin{abstract}
We define and study the cohomology theories associated to $A_\infty$--algebras and cyclic $A_\infty$--algebras equipped with an involution, generalising dihedral cohomology to the $A_\infty$ context. Such algebras arise, for example, as unoriented versions of topological conformal field theories. It is well known that Hochschild cohomology and cyclic cohomology govern, in a precise sense, the deformation theory of $A_\infty$--algebras and cyclic $A_\infty$--algebras and we give analogous results for the deformation theory in the presence of an involution. We also briefly discuss generalisations of these constructions and results to homotopy algebras over Koszul operads, such as $L_\infty$--algebras or $C_\infty$--algebras equipped with an involution.
\end{abstract}

\maketitle 
\tableofcontents

\section{Introduction}
In this note we shall study $A_\infty$--algebras and cyclic $A_\infty$--algebras (sometimes called Frobenius $A_\infty$--algebras) equipped with an involution, which we call involutive $A_\infty$--algebras and cyclic involutive $A_\infty$--algebras respectively. Examples of involutive $A_\infty$--algebras can often be found alongside standard examples of $A_\infty$--algebras arising in geometric contexts, when the underlying geometric object has an involution. For example, the de Rham cohomology of a closed oriented manifold carries the structure of a cyclic $A_\infty$--algebra. If the manifold is equipped with an involution then the de Rham cohomology is an involutive $A_\infty$--algebra and if the involution is orientation preserving then it is a cyclic involutive $A_\infty$--algebra. Since the identity is an example of such an involution, the de Rham cohomology of a closed oriented manifold could always be considered as a cyclic involutive $A_\infty$--algebra.

Costello \cite{costello2007:ribbondecomp,costello2007:tcftcalabiyau} showed that cyclic $A_\infty$--algebras are equivalent to open topological conformal field theories, in other words algebras over the modular operad of chains on certain moduli spaces of Riemann surfaces. Furthermore, given a cyclic $A_\infty$--algebra $A$ he constructs an open--closed topological conformal field theory whose open sector is $A$ and whose closed sector has underlying space $\hoch_\bullet(A,A)$, the Hochschild homology complex of $A$. It was shown in \cite{braun:moduliklein} that cyclic involutive $A_\infty$--algebras are equivalent to open Klein topological conformal field theories, in other words algebras over the modular operad of chains on certain moduli spaces of Klein surfaces. Therefore, following Costello, it is of interest to study the corresponding (co)homology theories associated to involutive $A_\infty$--algebras.

Hochschild cohomology is the natural cohomology theory associated to $A_\infty$--algebras and cyclic cohomology is the natural theory associated to cyclic $A_\infty$--algebras. We will recall this setup and compare it to the analogous theory for involutive $A_\infty$--algebras. In particular we will relate this to Loday's dihedral cohomology theory \cite{loday1987:dihedral,loday1992:cyclichombook} for involutive associative algebras, obtaining along the way a generalisation of dihedral cohomology to involutive $A_\infty$--algebras. This generalisation of the well-known construction of dihedral cohomology for involutive associative algebras to involutive $A_\infty$--algebras is so natural that it is no more complicated than the original theory. From the perspective of open Klein topological conformal field theories this also adds to the relevance of dihedral cohomology itself, exposing it as the `unoriented' version of cyclic cohomology.

It is well known that Hochschild cohomology and cyclic cohomology govern deformations of $A_\infty$--algebras and cyclic $A_\infty$--algebras respectively. This holds true for the corresponding involutive cohomology theories and in \autoref{sec:cyc} we shall explain more precisely what this means.

More generally, associated to any operad is a corresponding cohomology theory for algebras over (a cofibrant replacement for) that operad. The construction of attaching an involution to an operad outlined in \cite{braun:moduliklein} can be used to define and study involutive homotopy $\mathcal{P}$--algebras. In \autoref{sec:operads} we consider this more general setup for $\mathcal{P}$ a Koszul operad. In particular we will show that our notion of an involutive $\mathcal{P}$--algebra is a `correct' notion in the sense that it is equivalent to an algebra over a cofibrant replacement for the operad governing involutive $\mathcal{P}$--algebras. We should note however that our notion only weakens the multiplication up to homotopy but not the involution or bilinear form.

Many of the results here could be seen as an application of the general theory of algebraic operads and a detailed exposition of the general approach is contained in the book by Loday and Vallette \cite{lodayvallette2012:operadbook}. However, with the exception of the final section, we neither require nor desire heavy use of operad machinery since our aim is to understand in more detail the case for the specific operad $\ass$.

\subsection{Notation and conventions}
Throughout the paper $k$ will denote a field of characteristic zero.

Denote by $\dgvect$ the symmetric monoidal category of differential cohomologically $\mathbb{Z}$--graded $k$--linear vector spaces with symmetry isomorphism $s\co V\otimes W \rightarrow W\otimes V$ given by $s(v\otimes w) = (-1)^{\degree{v}\degree{w}} w \otimes v$. Here $\degree{v}$ and $\degree{w}$ are the degrees of the homogeneous elements $v$ and $w$. We denote the set of morphisms from $V$ to $W$, which are linear maps preserving the grading and the differentials, by $\Hom(V,W)$. This has the structure of a vector space. The category of graded vector spaces is a subcategory of $\dgvect$ by equipping graded vector spaces with the zero differential.

Let $\Sigma k$ be the one dimensional vector space concentrated in degree $-1$ with zero differential and let $\Sigma^{-1} k$ be the one dimensional vector space concentrated in degree $1$ with zero differential. We define the suspension functor by $V \mapsto \Sigma V = \Sigma k \otimes V$ and the desuspension functor by $V \mapsto \Sigma^{-1} V = \Sigma^{-1} k \otimes V$. By $\Sigma^n V$ we mean the $n$--fold suspension/desuspension of $V$.

Let $V,W$ be differential graded vector spaces. Denote by $\IHom(V,W)^n$ the vector space of homogeneous linear maps of degree $n$ (maps of graded vector spaces $f\co V\rightarrow \Sigma^n W$ preserving the grading but not necessarily the differential). Denote by $\IHom(V,W)=\bigoplus_n\Sigma^{-n}\IHom(V,W)^n$. This is a differential graded vector space with differential given by $dm = d_W\circ m - (-1)^{\degree{m}}m\circ d_V$ where $d_V$ and $d_W$ are the differentials on $V$ and $W$ respectively and $m$ is a homogeneous map of degree $\degree{m}$. Observe that $m$ is a chain map if both $\degree{m}=0$ and $dm=0$. In fact $\dgvect$ is then a symmetric monoidal closed category with internal $\Hom$ given by $\IHom$.

If one wishes to work with homologically graded objects then set $\Sigma k$ to be concentrated in degree $1$ and $\Sigma^{-1} k$ concentrated in degree $-1$. Then the definitions of the suspension and desuspension are the same as above. Similarly we can also work with supergraded objects, in which case $\Sigma=\Sigma^{-1}$. Everything presented will work in any of these gradings.

We will also need to consider \emph{formal} vector spaces. Precisely, a differential formal $\mathbb{Z}$--graded $k$--linear vector space $V$ is an inverse limit of finite dimensional $\mathbb{Z}$--graded $k$--linear vector spaces $V_i$, so that $V=\lim_{\leftarrow} V_i$, equipped with the inverse limit topology and a continuous differential. Morphisms between such spaces are required to be continuous linear maps preserving the grading and the differentials.

The \emph{completed tensor product} of two formal spaces $V=\lim_{\leftarrow} V_i$ and $W=\lim_{\leftarrow}W_j$ is the formal space $V \otimes W=\lim_{\leftarrow} V_i\otimes W_j$. Denote by $\fdgvect$ the symmetric monoidal category of differential formal $\mathbb{Z}$--graded $k$--linear vector spaces with symmetry isomorphism $s\co V \otimes W\rightarrow W \otimes V$ given by $s(v \otimes w) = (-1)^{\degree{v}\degree{w}} w \otimes v$. We have suspension and desuspension functors defined in the same way as before.

In particular given $V\in\dgvect$ then $V$ is the direct limit of its finite dimensional subspaces so $V=\lim_{\rightarrow}V_i$. We write $V^*\in\fdgvect$ for the space $V^*=\lim_{\leftarrow}\IHom(V_i, k)$. Similarly given $V\in\fdgvect$ with $V=\lim_{\leftarrow}V_i$ we write $V^*\in\dgvect$ for the space $V^*=\lim_{\rightarrow}\IHom(V_i,k)$. With this convention we have $V^{**}\cong V$ and $(V\otimes W)^*\cong V^*\otimes W^*$ for any pair $V$ and $W$ both in either $\dgvect$ or $\fdgvect$. In fact the functor $V\mapsto V^*$ is an anti-equivalence of symmetric monoidal categories. In particular an algebra in the category $\fdgvect$ is the same as a coalgebra in $\dgvect$.

Given $V=\lim_\leftarrow V_i\in\fdgvect$ and $W\in\dgvect$ define $V\otimes W$ to be the space $\lim_\leftarrow V_i\otimes W$, which in general is neither formal nor discrete. However, note that in this case $(V\otimes W)^0$ is the space of linear maps from $V^*$ to $W$ preserving the grading.

For clarity we will write $\Sigma V^*$ to mean $\Sigma (V^*)$. There is a natural isomorphism $(\Sigma V)^* \cong \Sigma^{-1} V^*$.

Note that in general we will not require algebras to be unital unless stated. By an augmented associative or commutative algebra it is meant a unital algebra $A$ equipped with an algebra map $A\rightarrow k$. The augmentation ideal $A_+$ is the kernel of this map. By a \emph{formal} associative/commutative algebra we mean a commutative algebra which is an inverse limit of finite dimensional \emph{nilpotent} associative/commutative algebras. Given an augmented formal associative/commutative algebra $R$ then the maximal ideal is $R_+$.

Given $W\in\fdgvect$ we write $\widehat{T}W$ for the free formal augmented differential graded associative algebra generated by $W$ given explicitly by the completed tensor algebra
\[
\widehat{T}W = \prod_{n=0}^\infty W^{\otimes n} = k \times W \times (W\otimes W) \times \dots
\]
We write $\widehat{T}_{\geq i}W$ for the subalgebra of elements of order $i$ and above. This gives a decreasing filtration of this algebra.

\section{Involutive Hochschild cohomology of involutive \texorpdfstring{$A_\infty$--}{A-infinity }algebras}\label{sec:noncyc}
We begin by briefly recalling the basic definitions concerning $A_\infty$--algebras. For a gentler introduction see \cite{keller2001:introtoainfinity,keller2006:ainfinityalgebrasmodules}.

\begin{definition}
Let $A$ be an graded associative algebra. A \emph{graded derivation} of $A$ is a graded map $f\co A\rightarrow A$ such that for any $x,y\in A$
\[
f(xy) = f(x)y + (-1)^{\degree{f}\degree{x}}xf(y).
\]
The space spanned by all graded derivations forms a Lie subalgebra of the space $\IHom(A,A)$ of graded linear maps with the commutator bracket and is denoted by $\Der(A)$.
\end{definition}

\begin{definition}
Let $V$ be a graded vector space. An \emph{$A_\infty$--algebra structure} on $V$ is a derivation $m\co \widehat{T}_{\geq 1}\Sigma^{-1}V^*\rightarrow \widehat{T}_{\geq 1}\Sigma^{-1}V^*$ of degree one such that $m^2 = 0$.
\end{definition}

\begin{remark}
Since we are considering the \emph{completed} tensor algebra, which is a formal vector space, maps and derivations are of course required to be continuous.
\end{remark}

Recall that such a derivation $m$ is determined completely by its restriction to $\Sigma^{-1}V^*$. Denote by $m_n\co\Sigma^{-1}V^*\rightarrow (\Sigma^{-1}V^*)^{\otimes n}$ the order $n$ part of this restriction so that $m = m_1 + m_2 + \dots$ on the subspace $\Sigma^{-1}V^*$. Such a collection of $m_n$ is equivalent to a collection of maps $\hat{m}_n\co V^{\otimes n} \rightarrow V$ of degrees $2-n$ satisfying the usual $A_\infty$--conditions.

Note that $m_1^2 = 0$ so $(V,\hat{m}_1)$ is a differential graded vector space. If $m_n = 0$ for $n > 2$ then $\hat{m}_2$ is an associative product on $V$ respecting the differential and so a differential graded associative algebra is a special case of an $A_\infty$--algebra.

\begin{definition}
Let $(V,m)$ and $(W,m')$ be $A_\infty$--algebras. Then an \emph{$A_\infty$--morphism} of $A_\infty$--algebras is a map $\phi$ of associative algebras $\phi\co\widehat{T}_{\geq 1}\Sigma^{-1}W^*\rightarrow \widehat{T}_{\geq 1}\Sigma^{-1}V^*$ such that $m \circ\phi = \phi\circ m'$.
\end{definition}

\begin{definition}
Let $(V,m)$ be an $A_\infty$--algebra. Then the space of derivations $\Der(\widehat{T}\Sigma^{-1}V^*)$ is naturally a differential graded Lie algebra with bracket the commutator bracket and differential given by $d(\xi)=[m,\xi]$.

The \emph{Hochschild cohomology complex} of $V$ with coefficients in itself is the differential graded vector space $\hoch^{\bullet}(V,V) = \Sigma^{-1}\Der(\widehat{T}\Sigma^{-1}V^*)$. The cohomology of this will be denoted by $\Hhoch^{\bullet}(V,V)$.
\end{definition}

In the case that $V$ is an associative algebra then note that this coincides with the classical definition of the Hochschild cohomology. This is also the reason for the desuspension in this definition.

There is a natural decreasing filtration of Lie algebras (so that filtration degree is preserved by the bracket) $F_{-1}\supset F_0\supset \dots $ on $\Der(\widehat{T}\Sigma^{-1}V^*)$, defined by $F_n=\Der(\widehat{T}\Sigma^{-1}V^*)_{\geq n}$ where $m\in \Der(\widehat{T}\Sigma^{-1}V^*)_{\geq n}$ if and only if $m_k = 0$ for all $k<n+1$. Note that $\Der(\widehat{T}\Sigma^{-1}V^*)_{\geq -1} = \Der(\widehat{T}\Sigma^{-1}V^*)$ and $\Der(\widehat{T}\Sigma^{-1}V^*)_{\geq 0} = \Der(\widehat{T}_{\geq 1}\Sigma^{-1}V^*)$. If $V$ is concentrated in degree $0$ then this coincides with the filtration by degree. We set $\hoch^\bullet(V,V)_{\geq n} = \Sigma^{-1}\Der(\widehat{T}\Sigma^{-1}V^*)_{\geq n-1}$.

\subsection{Involutive \texorpdfstring{$A_\infty$--}{A-infinity }algebras}
\Needspace*{4\baselineskip}

\begin{definition}
Let $V$ be a graded vector space. An involution on $V$ is a map $v\mapsto v^*$ with $(v^*)^*=v$. An \emph{involutive differential graded associative algebra} is a differential graded associative algebra $A$ with an involution satisfying $(xy)^* = (-1)^{\degree{x} \degree{y}}y^*x^*$ and $d(x)^* = d(x^*)$.
\end{definition}

If $V$ has an involution then $\Sigma V$ and $\Sigma^{-1} V$ do in the obvious way. $V^*$ also has an involution defined by $\phi^*(v) = -\phi(v^*)$ for $\phi\in V^*$ (note the appearance of the minus sign here).

Let $W$ be an graded vector space with an involution. Then $W^{\otimes n}$ admits an involution defined by
\[
(w_1\otimes w_2 \otimes \dots \otimes w_n)^* = (-1)^{\epsilon} w_n^*\otimes\dots\otimes w_2^* \otimes w_1^*
\]
where $\epsilon = \sum_{i=1}^n \degree{w_i}\left ( \sum_{j=i+1}^n \degree{w_j} \right )$ arises from permuting the $w_i$ with degrees $\degree{w_i}$. It follows that $\widehat{T}\Sigma^{-1}V^*$ has an involution induced by that on $V$ making it into an involutive graded associative algebra.

\begin{definition}
Let $V$ be a graded vector space with an involution. An \emph{involutive $A_\infty$--algebra structure} on $V$ is a derivation $m\co \widehat{T}_{\geq 1} \Sigma^{-1} V^* \rightarrow \widehat{T}_{\geq 1} \Sigma^{-1} V^*$ of degree one such that $m^2=0$ and $m$ preserves the involution: $m(x^*) = m(x)^*$.
\end{definition}

The requirement $m(x^*) = m(x)^*$ can be unwrapped in terms of the $\hat{m}_n\co V^{\otimes n}\rightarrow V$ and the involution on $V$ to obtain
\[
\hat{m}_n(x_1,\dots , x_n)^* = (-1)^{\epsilon}(-1)^{n(n+1)/2-1} \hat{m}_n(x_n^*,\dots,x_1^*)
\]
where $\epsilon = \sum_{i=1}^n \degree{x_i}\left ( \sum_{j=i+1}^n \degree{x_j} \right )$ arises from permuting the $x_i\in V$ with degrees $\degree{x_i}$. In particular if $m_n = 0 $ for $n > 2$ then this corresponds to the structure of an involutive differential graded associative algebra, which is thus a special case of an involutive $A_\infty$--algebra.

\begin{definition}
Let $(V,m)$ and $(W,m')$ be involutive $A_\infty$--algebras. Then an \emph{involutive $A_\infty$--morphism} of involutive $A_\infty$--algebras is a map $\phi$ of associative algebras $\phi\co\widehat{T}_{\geq 1}\Sigma^{-1}W^*\rightarrow \widehat{T}_{\geq 1}\Sigma^{-1}V^*$ such that $m \circ\phi = \phi\circ m'$ and $\phi$ preserves the involution: $\phi(x^*) = \phi(x)^*$.
\end{definition}

The requirement that $\phi$ preserves the involution is, of course, the same as requiring it to be a map of involutive associative algebras.

\begin{definition}
Let $(V,m)$ be an involutive $A_\infty$--algebra. Then the subspace of derivations $\Der_+(\widehat{T}\Sigma^{-1}V^*)\subset \Der(\widehat{T}\Sigma^{-1}V^*)$ preserving the involution is naturally a differential graded Lie subalgebra with bracket the commutator bracket and differential given by $d(\xi)=[m,\xi]$.

The \emph{involutive Hochschild cohomology complex} of $V$ with coefficients in itself is the differential graded vector space $\hoch^{\bullet}_+(V,V) = \Sigma^{-1}\Der_+(\widehat{T}\Sigma^{-1}V^*)$. The cohomology of this will be denoted by $\Hhoch^{\bullet}_+(V,V)$.
\end{definition}

\subsection{Decomposition of Hochschild cohomology}
Let $V$ be a graded vector space with an involution.

\begin{definition}\label{def:skewinvlie}
A \emph{skew-involutive differential graded Lie algebra} is a differential graded Lie algebra $\mathfrak{g}$ with an involution satisfying $[x,y]^* = [x^*,y^*]$ and $d(x)^* = d(x^*)$.
\end{definition}

\begin{proposition}
For $\xi\in\Der(\widehat{T}\Sigma^{-1}V^*)$ define $\xi^*$ by $\xi^*(x) = \xi(x^*)^*$. Then $\xi^*\in\Der(\widehat{T}\Sigma^{-1}V^*)$ and furthermore this makes the Lie algebra $\Der(\widehat{T}\Sigma^{-1}V^*)$ into a skew-involutive Lie algebra, by which it is meant that $[\xi,\eta]^* = [\xi^*,\eta^*]$.
\end{proposition}

\begin{proof}
Straightforward calculation shows that $\xi^*$ is indeed a derivation. The result now follows from the observation that $\xi(\eta(x^*))^* = \xi^*(\eta^*(x))$.
\end{proof}

\begin{remark}
If one instead preferred to define the involutive Hochschild cohomology complex as the \emph{quotient} of the usual Hochschild cohomology complex by the action of $\mathbb{Z}_2$ given by $\xi\mapsto \xi^*$ this is of course equivalent to the definition used above by the isomorphism of invariants and coinvariants.
\end{remark}

Note that $\Der_+(\widehat{T}\Sigma^{-1}V^*)$ is the eigenspace of the eigenvalue $+1$ of this involution. So denote by $\Der_-(\widehat{T}\Sigma^{-1}V^*)$ the eigenspace of the eigenvalue $-1$. For $\xi\in\Der(\widehat{T}\Sigma^{-1}V^*)$ denote by $\xi\mapsto\xi^+$ and $\xi\mapsto\xi^-$ the projections onto these eigenspaces given by $\xi^+ = 1/2(\xi+\xi^*)$ and $\xi^- = 1/2(\xi - \xi^*)$. Then $\xi=\xi^+ + \xi^-$ and as graded vector spaces $\Der(\widehat{T}\Sigma^{-1}V^*)=\Der_+(\widehat{T}\Sigma^{-1}V^*)\oplus \Der_-(\widehat{T}\Sigma^{-1}V^*)$. Note that this is not a decomposition of Lie algebras, however
\[
[\Der_-(\widehat{T}\Sigma^{-1}V^*),\Der_-(\widehat{T}\Sigma^{-1}V^*)] \subset \Der_+(\widehat{T}\Sigma^{-1}V^*).
\]
Now let $(V,m)$ be an involutive $A_\infty$--algebra. Then since $[m,\xi]^* = [m,\xi^*]$ it follows that this decomposition is in fact a decomposition of \emph{differential} graded vector spaces.

\begin{definition}
The \emph{skew-involutive Hochschild cohomology complex} of an involutive $A_\infty$--algebra $(V,m)$ with coefficients in itself is the differential graded vector space $\hoch^{\bullet}_-(V,V) = \Sigma^{-1}\Der_-(\widehat{T}\Sigma^{-1}V^*)$. The cohomology of this will be denoted by $\Hhoch^{\bullet}_-(V,V)$.
\end{definition}

\begin{theorem}\label{thm:hochschilddecomposition}
For an involutive $A_\infty$--algebra $(V,m)$ the Hochschild cohomology of $V$ decomposes as $\Hhoch^\bullet(V,V)\cong \Hhoch^{\bullet}_+(V,V)\oplus\Hhoch^{\bullet}_-(V,V)$.
\qed
\end{theorem}

\begin{remark}
From the perspective of deformation theory, which we will elaborate upon later, for an involutive $A_\infty$--algebra $(V,m)$ general theory tells us that $\Hhoch_+^\bullet(V,V)$ governs involutive $A_\infty$--deformations of $m$ and $\Hhoch^\bullet(V,V)$ governs $A_\infty$--deformations of $m$. Therefore in a certain sense $\Hhoch_-^\bullet(V,V)$ should measure the difference between the involutive and usual deformation theory.
\end{remark}

\section{Dihedral cohomology of cyclic involutive \texorpdfstring{$A_\infty$--}{A-infinity }algebras}\label{sec:cyc}
We begin by recalling the cyclic cohomology of $A_\infty$--algebras and how it relates to cyclic $A_\infty$--algebras.

\subsection{Cyclic \texorpdfstring{$A_\infty$--}{A-infinity }algebras}
\Needspace*{4\baselineskip}

\begin{proposition}\label{prop:cyclicdifferential}
Let $(V, m)$ be an $A_\infty$--algebra. Denote by $\CC^{\bullet}(V)$ the graded vector space
\[
\CC^{\bullet}(V) = \Sigma \prod_{i=1}^{\infty}[(\Sigma^{-1}V^*)^{\otimes i}]_{\mathbb{Z}_i}
\]
where $\mathbb{Z}_i$ is the cyclic group of order $i$ acting in the obvious way. Then the derivation $m$ on $\widehat{T}\Sigma^{-1}V^*$ induces a well defined derivation on the quotient $\CC^{\bullet}(V)$.
\end{proposition}

\begin{proof}
Let $M$ be the subspace of $\widehat{T}\Sigma^{-1}V^*$ of convergent sums of elements of the form $ab-(-1)^{\degree{a}\degree{b}}ba$ (in other words the subspace of commutators). Then $\CC^\bullet(V)=\widehat{T}\Sigma^{-1}V^*/M$. A straightforward calculation using the fact that $m$ is a degree one derivation shows that $m(M)\subset M$ as required.
\end{proof}

\begin{definition}
Let $(V,m)$ be an $A_\infty$--algebra. The \emph{cyclic cohomology complex} of $V$ is the differential graded vector space $\CC^{\bullet}(V)$ with differential induced by $m$ as in \autoref{prop:cyclicdifferential}. The cohomology of this will be denoted by $\HCC^{\bullet}(V)$.
\end{definition}

Let $V$ be a graded vector space with a symmetric bilinear form $\langle -, - \rangle \co V\otimes V \rightarrow \Sigma^d k$ of degree $d$. A derivation $m\in\Der(\widehat{T}\Sigma^{-1}V^*)$ is called a \emph{cyclic derivation} if the maps $\hat{m}_n\co V^{\otimes n}\rightarrow V$ satisfy
\begin{equation}\label{eq:cycliccondition}
\langle \hat{m}_n(x_1,\dots , x_n), x_{n+1} \rangle = (-1)^\epsilon(-1)^n \langle \hat{m}_n(x_{n+1},\dots, x_{n-1}), x_n \rangle
\end{equation}
where $\epsilon = \degree{x_{n+1}}\sum_{i=1}^n \degree{x_i}$ arises from permuting the $x_i\in V$. The subspace of cyclic derivations will be denoted by $\Der_\cycl(\widehat{T}\Sigma^{-1}V^*)$. It is a Lie subalgebra.

The bilinear form on $V$ is non-degenerate if the map $V\rightarrow \Sigma^d V^*$ given by $v\mapsto \langle v, - \rangle$ is an isomorphism. This yields a degree $-d$ non-degenerate bilinear form on $V^*$ denoted by $\langle -, - \rangle^{-1}\co V^*\otimes V^* \rightarrow \Sigma^{-d}V^*$ which is symmetric if $d$ is even and anti-symmetric if $d$ is odd.

\begin{definition}
Let $V$ be a graded vector space with a symmetric bilinear form. A \emph{cyclic $A_\infty$--algebra structure} on $V$ is a cyclic derivation $m\in\Der^\cycl(\widehat{T}\Sigma^{-1}V^*)$ of degree one such that $m^2=0$.
\end{definition}

\begin{remark}
A differential graded associative algebra with a symmetric bilinear form satisfying $\langle ab, c \rangle = \langle a, bc \rangle $ is a special case of a cyclic $A_\infty$--algebra.
\end{remark}

Note that a symmetric bilinear form on $V$ of degree $d$ is represented by a degree $d$ element in $V^*\otimes V^*$, or equivalently a degree $d+2$ element in $\Sigma^{-1}V^*\otimes \Sigma^{-1}V^*$.

\begin{definition}
Let $(V,m)$ and $(W,m')$ be cyclic $A_\infty$--algebras with degree $d$ symmetric bilinear forms. Then a \emph{cyclic $A_\infty$--morphism} is a map $\phi$ of $A_\infty$--algebras such that $\phi(\omega') = \omega$ where $\omega\in\Sigma^{-1}V^*\otimes\Sigma^{-1}V^*$ and $\omega'\in\Sigma^{-1}W^*\otimes \Sigma^{-1}W^*$ are the degree $d+2$ elements representing the bilinear forms.
\end{definition}

\begin{remark}\label{rem:preservebilinear}
An $A_\infty$--morphism $\phi$ from $V$ to $W$ is determined by a collection of maps $\hat{\phi}_n\co V^{\otimes n} \rightarrow W$ of degrees $1-n$ satisfying certain compatibility conditions with the $A_\infty$--structures. The condition $\phi(\omega')=\omega$ for a cyclic $A_\infty$--morphism can then be restated explicitly as
\begin{gather*}
\langle \hat{\phi}_1(a_1),\hat{\phi}_1(a_2) \rangle = \langle a_1, a_2 \rangle\\
\sum_{i+j=n} (-1)^\epsilon \langle \hat{\phi}_i(a_1,\dots, a_i), \hat{\phi}_j(a_{i+1},\dots, a_{n} ) \rangle = 0
\end{gather*}
where $\epsilon = (1-j)(\degree{a_1}+\dots +\degree{a_i})$ for $a_i\in V$.
\end{remark} 

\begin{theorem}\label{thm:cycder}
Let $(V,m)$ be a cyclic $A_\infty$--algebra with a \emph{non-degenerate} symmetric bilinear form. Then as complexes $\Sigma^{d+1}\CC^\bullet(V) \cong \Der^\cycl(\widehat{T}\Sigma^{-1}V^*)$ where the right hand side is equipped with the differential given by $d(\xi) = [m,\xi]$.
\end{theorem}

\begin{proof}
Since $\IHom(\Sigma^{-1}V^*, (\Sigma^{-1}V^*)^{\otimes n})\cong \Sigma V \otimes (\Sigma^{-1}V^*)^{\otimes n}$ then composing with the isomorphism $V \rightarrow \Sigma^{d}V^*$ on the first tensor factor gives an isomorphism
\[
f\co\IHom(\Sigma^{-1}V^*, (\Sigma^{-1}V^*)^{\otimes n})\cong \Sigma^{d+2} (\Sigma^{-1}V^*)^{\otimes n+1}.
\]
Given $\xi_n\in \IHom(\Sigma^{-1}V^*, (\Sigma^{-1}V^*)^{\otimes n})$ then $\hat{\xi}_n$ satisfies \autoref{eq:cycliccondition} if and only if $f(\xi_n)$ is an invariant with respect to the cyclic action on $\Sigma^{d+2} (\Sigma^{-1}V^*)^{\otimes n+1}$. By the isomorphism of coinvariants and invariants there is therefore an isomorphism of graded vector spaces $\Der^\cycl(\widehat{T}\Sigma^{-1}V^*)\cong \Sigma^{d+1}\CC^\bullet(V)$. It remains to verify that the differential coincides, in other words that $f([m,\xi]) = m(f(\xi))$, which is a straightforward check left to the reader.
\end{proof}

\begin{remark}
Kontsevich's formal noncommutative symplectic geometry \cite{kontsevich1993:ncsg} says that by regarding the bilinear form on $V$ as a symplectic structure then the underlying space of $\CC^{\bullet}(V)$ can be understood as the space of noncommutative Hamiltonians, and $\Der^{\cycl}(\widehat{T}\Sigma^{-1}V^*)$ can be understood as the space of symplectic vector fields. This gives a rather enlightening view of \autoref{thm:cycder}. For a detailed account of this point of view see \cite{hamiltonlazarev:homalgsncg}.
\end{remark}

\begin{remark}\label{rem:cyclie}
It follows that $\Sigma^{d+1}\CC^\bullet(V)$ has the structure of a differential graded Lie algebra. The Lie bracket can be described explicitly on the summands by the formula
\[
[a_1\dots a_n, b_1\dots b_m ] = (-1)^p\sum_{i=1}^{n}\sum_{j=1}^{m}(-1)^{\epsilon}\langle a_i, b_j \rangle^{-1}  a_{i+1}\dots a_1\ldots a_{i-1} b_{j+1}\dots b_1\dots b_{j-1}
\]
where $\epsilon$ arises from permuting the $a_i\in\Sigma^{-1}V^*$ and $b_i\in\Sigma^{-1}V^*$ and $p=d\sum_i \degree{a_i}$.
\end{remark}

We set $\CC^\bullet(V)_{\geq n} = \Sigma \prod_{i=n+1}^{\infty}[(\Sigma^{-1}V^*)^{\otimes i}]_{\mathbb{Z}_i}$. If $V$ is a cyclic $A_\infty$--algebra with a non-degenerate bilinear form then $\Sigma^{d+1}\CC^\bullet(V)_{\geq n} \cong \Der^\cycl(\widehat{T}\Sigma^{-1}V^*)_{\geq n-1}$.

\subsection{Cyclic involutive \texorpdfstring{$A_\infty$--}{A-infinity }algebras and decomposition of cyclic cohomology}
\Needspace*{4\baselineskip}

\begin{definition}\label{def:dihedralactions}
Let $W$ be a graded vector space with an involution. Denote by $D_n$ the dihedral group of order $2n$, presented by $D_n = \langle r, s \mid r^n = s^2 = 1, srs^{-1}=r^{-1} \rangle $. Then there are the following two actions of $D_n$ on $W^{\otimes n}$.
\begin{enumerate}
\item \label{dihedralaction}The \emph{dihedral action} is defined by
\begin{align*}
r(w_1\otimes w_2 \otimes \dots \otimes w_n) &= (-1)^{\epsilon} w_n\otimes w_1 \otimes \dots\otimes w_{n-1}\\
s(w_1\otimes w_2 \otimes \dots \otimes w_n) &= (w_1\otimes w_2 \otimes \dots \otimes w_n)^*
\end{align*}
\item\label{skewdihedralaction} The \emph{skew-dihedral action} is defined by
\begin{align*}
r(w_1\otimes w_2 \otimes \dots \otimes w_n) &= (-1)^{\epsilon} w_n\otimes w_1 \otimes \dots\otimes w_{n-1}\\
s(w_1\otimes w_2 \otimes \dots \otimes w_n) &= -(w_1\otimes w_2 \otimes \dots \otimes w_n)^*
\end{align*}
\end{enumerate}
Here $\epsilon= \degree{w_n}\sum_{i=1}^{n-1}\degree{w_i}$ arises, as usual, from permuting the $w_i\in W$.
\end{definition}

\begin{proposition}\label{prop:dihedraldifferential}
Let $(V, m)$ be an involutive $A_\infty$--algebra.
\begin{itemize}
\item Denote by $\CD^{\bullet}_+(V)$ the graded vector space
\[
\CD^{\bullet}_+(V)= \Sigma \prod_{i=1}^{\infty}[ (\Sigma^{-1} V^*)^{\otimes i} ]_{D_{i}}
\]
where $D_i$ is the dihedral group of order $2i$ acting by the dihedral action \ref{dihedralaction} of \autoref{def:dihedralactions}. Then the derivation $m$ on $\widehat{T}\Sigma^{-1}V^*$ induces a well defined derivation on the quotient $\CD^\bullet_+(V)$.
\item Denote by $\CD^{\bullet}_-(V)$ the graded vector space
\[
\CD^\bullet_-(V) = \Sigma \prod_{i=1}^{\infty}[ (\Sigma^{-1} V^*)^{\otimes i} ]_{D_{i}}
\]
where $D_i$ is the dihedral group of order $2i$ acting by the skew-dihedral action \ref{skewdihedralaction} of \autoref{def:dihedralactions}. Then the derivation $m$ on $\widehat{T}\Sigma^{-1}V^*$ induces a well defined derivation on the quotient $\CD^\bullet_-(V)$.
\end{itemize}
\end{proposition}

\begin{proof}
Let $M$ be the subspace of $\widehat{T}\Sigma^{-1}V^*$ of convergent sums of elements of the forms $ab - (-1)^{\degree{a}\degree{b}}ba$ and $a - a^*$. Then $\CD^\bullet_+(V) = \Sigma\widehat{T}\Sigma^{-1}V^*/M$. Since $m$ is involutive $m(a-a^*) = m(a)-m(a)^* \in M$ and together with the proof of \autoref{prop:cyclicdifferential} this means $m(M)\subset M$ and the first part follows. The second part is essentially the same.
\end{proof}

\begin{definition}
Let $(V,m)$ be an involutive $A_\infty$--algebra.
\begin{itemize}
\item The \emph{dihedral cohomology complex} of $V$ is the differential graded vector space $\CD^\bullet_+(V)$ with differential induced by $m$ as in \autoref{prop:dihedraldifferential}. The cohomology of this complex will be denoted by $\HCD^{\bullet}_+(V)$.
\item The \emph{skew-dihedral cohomology complex} of $V$ is the differential graded vector space $\CD^\bullet_-(V)$ with differential induced by $m$ as in \autoref{prop:dihedraldifferential}. The cohomology of this complex will be denoted by $\HCD^\bullet_-(V)$.
\end{itemize}
\end{definition}

\begin{theorem}\label{thm:cyclicdecomposition}
For an involutive $A_\infty$--algebra $(V, m)$ the cyclic cohomology of $V$ decomposes as $\HCC^\bullet(V) \cong \HCD^\bullet_+(V) \oplus \HCD^\bullet_-(V)$.
\end{theorem}

\begin{proof}
Since $D_n = \mathbb{Z}_2\ltimes \mathbb{Z}_n$, the complexes $\CD^\bullet_+(V)$ and $\CD^\bullet_-(V)$ are the quotients of $\CC^\bullet(V)$ by two different actions of $\mathbb{Z}_2$ arising from the involution on $V$. By the isomorphism of coinvariants with invariants these spaces can be identified with the eigenspaces of the eigenvalues $+1$ and $-1$ of the involution on $\CC^\bullet(V)$ and the result follows as for \autoref{thm:hochschilddecomposition}.
\end{proof}

\begin{proposition}
Let $V$ be a graded vector space with an involution and a symmetric bilinear form such that $\langle x^*, y^* \rangle = \langle x, y \rangle$. Then $\Der^\cycl(\widehat{T}\Sigma^{-1} V^*)$ is a skew-involutive Lie subalgebra of $\Der(\widehat{T}\Sigma^{-1} V^*)$.
\end{proposition}

\begin{proof}
Let $m\in \Der^\cycl(\widehat{T}\Sigma^{-1} V^*)$ be a cyclic derivation. Then a straightforward calculation gives
\begin{align*}
\langle \hat{m}^*_n(x_1,\dots , x_n), x_{n+1} \rangle &= (-1)^{\epsilon}(-1)^{n(n+1)/2 - 1} \langle \hat{m}_n(x_n^*,\dots,x_1^*)^*, x_{n+1} \rangle
\\
&= (-1)^{\epsilon}(-1)^{n(n+1)/2 - 1} \langle \hat{m}_n(x_n^*,\dots,x_1^*), x_{n+1}^* \rangle
\\
&=  (-1)^{\epsilon'}(-1)^{n(n+1)/2 - 1}(-1)^n \langle \hat{m}_n(x_{n-1}^*,\dots, x_{n+1}^*), x_n^* \rangle
\\
&= (-1)^{\epsilon'}(-1)^{n(n+1)/2 - 1}(-1)^n \langle \hat{m}_n(x_{n-1}^*,\dots, x_{n+1}^*)^*, x_n \rangle
\\
&= (-1)^{\epsilon''}(-1)^n \langle \hat{m}_n^*(x_{n+1},\dots, x_{n-1}), x_n \rangle
\end{align*}
where $\epsilon$, $\epsilon'$ and $\epsilon''$ arise from the Koszul sign rule permuting the $x_i$, remembering that $\degree{x_i^*} = \degree{x_i}$. Therefore $m^*$ is also a cyclic derivation as required. 
\end{proof}

\begin{corollary}
The space $\Der^{\cycl}_+(\widehat{T}\Sigma^{-1}V^*)=\Der^\cycl(\widehat{T}\Sigma^{-1}V^*)\cap\Der_+(\widehat{T}\Sigma^{-1}V^*)$ of cyclic derivations preserving the involution is a Lie subalgebra and as graded vector spaces $\Der^\cycl(\widehat{T}\Sigma^{-1}V^*) = \Der^{\cycl}_+(\widehat{T}\Sigma^{-1}V^*)\oplus\Der^{\cycl}_-(\widehat{T}\Sigma^{-1}V^*)$.
\qed
\end{corollary}

\begin{definition}
Let $V$ be a graded vector space with an involution and a symmetric bilinear form such that $\langle x^*, y^* \rangle = \langle x, y \rangle$. A \emph{cyclic involutive $A_\infty$--algebra structure} on $V$ is a cyclic involutive derivation $m \in \Der^\cycl_+(\widehat{T} \Sigma^{-1} V^*)$ of degree one such that $m^2 = 0$.
\end{definition}

\begin{remark}
An involutive differential graded associative algebra with a symmetric bilinear form satisfying $\langle a^*, b^* \rangle = \langle a, b \rangle$ and $\langle ab, c \rangle = \langle a, bc \rangle$ is a special case of a cyclic involutive $A_\infty$--algebra.
\end{remark}

\begin{definition}
Let $(V,m)$ and $(W,m')$ be cyclic involutive $A_\infty$--algebras with degree $d$ symmetric bilinear forms. Then a \emph{cyclic involutive $A_\infty$--morphism} is a map $\phi$ of involutive $A_\infty$--algebras which is also a map of cyclic $A_\infty$--algebras.
\end{definition}

\begin{theorem}
Let $(V, m)$ be a cyclic involutive $A_\infty$--algebra with a \emph{non-degenerate} symmetric bilinear form. Then as complexes
\begin{itemize}
\item $\Sigma^{d+1}\CD^\bullet_+(V) \cong \Der^\cycl_+(\widehat{T} \Sigma^{-1} V^*)$
\item $\Sigma^{d+1}\CD^\bullet_-(V) \cong \Der^\cycl_-(\widehat{T} \Sigma^{-1} V^*)$
\end{itemize}
where $\Der^\cycl_+(\widehat{T} \Sigma^{-1} V^*)$ and $\Der^\cycl_-(\widehat{T} \Sigma^{-1} V^*)$ are each equipped with the differential given by $d(\xi)=[m,\xi]$.
\end{theorem}

\begin{proof}
By \autoref{thm:cycder} $\Sigma^{d+1} \CC^\bullet(V) \cong \Der^\cycl(\widehat{T}\Sigma^{-1}V^*)$ as complexes. Furthermore it is clear this isomorphism also preserves the involution. As in the proof of \autoref{thm:cyclicdecomposition} $\CD^\bullet_+(V)$ and $\CD^\bullet_-(V)$ can be identified with the eigenspaces of the eigenvalues $+1$ and $-1$ of the involution on $\CC^\bullet(V)$, which correspond under this isomorphism to $\Der^\cycl_+(\widehat{T}\Sigma^{-1}V^*)$ and $\Der^\cycl_-(\widehat{T}\Sigma^{-1}V^*)$ respectively.
\end{proof}

\begin{remark}
It follows from \autoref{rem:cyclie} that $\Sigma^{d+1}\CD^\bullet_+(V)$ is a differential graded Lie subalgebra of $\Sigma^{d+1}\CC^\bullet(V)$.
\end{remark}

\section{Deformation theory of involutive \texorpdfstring{$A_\infty$--}{A-infinity }algebras}\label{sec:deftheory}
Just as the Hochschild cohomology and cyclic cohomology in a certain sense govern deformations of $A_\infty$--algebras and cyclic $A_\infty$--algebras, so we will see that the involutive Hochschild cohomology and dihedral cohomology govern deformations of involutive $A_\infty$--algebras and cyclic involutive $A_\infty$--algebras.

There is much well established general theory concerning deformation functors and their representability by differential graded Lie algebras so this section will be succinct and not at all comprehensive, particularly since this section will likely be self-evident to experts. If the reader is interested in a more general approach for homotopy algebras over arbitrary operads see \cite{hinich2004:deformationsofhomalgs,hinich1997:homalg}.

\subsection{The Maurer--Cartan moduli set}
We first recall the basic definition of a Maurer--Cartan element in a Lie algebra.

\begin{definition}
Let $\mathfrak{g}$ be a differential graded Lie algebra. A \emph{Maurer--Cartan element} in $\mathfrak{g}$ is a degree one element $\xi\in\mathfrak{g}$ satisfying the Maurer--Cartan equation
\[
d\xi + \frac{1}{2}[\xi, \xi]=0.
\]
We denote the set of Maurer--Cartan elements in $\mathfrak{g}$ by $\MC(\mathfrak{g})$.
\end{definition}

Since a map of $\mathfrak{g}\rightarrow\mathfrak{h}$ takes Maurer--Cartan elements to Maurer--Cartan elements we see that $\MC$ defines a functor on differential graded Lie algebras.

\begin{example}
Let $V$ be a graded vector space. Given an odd derivation $m\in\Der(\widehat{T}\Sigma^{-1}V^*)$ then $\frac{1}{2}[m,m]=m^2$. Therefore $A_\infty$--structures on $V$ are in one-to-one correspondence with Maurer--Cartan elements in the graded Lie algebra $\Der(\widehat{T}_{\geq 1}\Sigma^{-1}V^*)$. Similarly involutive, cyclic and cyclic involutive $A_\infty$--structures are in correspondence with Maurer--Cartan elements in the graded Lie algebras $\Der_+(\widehat{T}_{\geq 1}\Sigma^{-1}V^*)$, $\Der^\cycl(\widehat{T}_{\geq 1}\Sigma^{-1}V^*)$ and $\Der^\cycl_+(\widehat{T}_{\geq 1}\Sigma^{-1}V^*)$ respectively.
\end{example}

Given a Lie algebra $\mathfrak{g}$ define the ideals $[\mathfrak{g}]^n$ recursively by $[\mathfrak{g}]^1 = \mathfrak{g}$ and $[\mathfrak{g}]^{n} = [\mathfrak{g},\mathfrak{g}^{n-1}]$. Then $\mathfrak{g}$ is called \emph{nilpotent} if the \emph{descending central series} $[\mathfrak{g}]^1\supset [\mathfrak{g}]^2\supset [\mathfrak{g}]^3\supset\dots$ stabilises at $0$. Note that in the case $\mathfrak{g}$ is finite dimensional this is equivalent to the definition that for every $\xi\in\mathfrak{g}$, $\ad_\xi$ is nilpotent.

Now assume $\mathfrak{g}$ is \emph{pronilpotent}, by which it is meant that it is an inverse limit of nilpotent algebras. Recall that for every such Lie algebra there is an associative product $\bullet\co\mathfrak{g}\times\mathfrak{g}\rightarrow \mathfrak{g}$ given by the Baker--Campbell--Hausdorff formula which is functorial (given $f\co\mathfrak{g}\rightarrow\mathfrak{h}$ then $f(x\bullet y) = f(x)\bullet f(y)$) and for any unital associative algebra $A$ with pronilpotent ideal $I$ it holds for any $a,b\in I$ that $e^a e^b = e^{a\bullet b}$ where $e^a = \sum_{n\geq 0} \frac{a^n}{n!}\in A$ and $\bullet$ is taken with respect to the commutator Lie bracket on $A$. A property of $\bullet$ is that for any $x,y\in\mathfrak{g}$ if $[x,y]=0$ then $x\bullet y = x + y$.

Define the group $\exp(\mathfrak{g}) = \{e^x : x\in\mathfrak{g} \}$ with product defined as $e^x\cdot e^y = e^{x\bullet y}$. The identity is $1=e^0$ and $e^x\cdot e^{-x} = e^{-x}\cdot e^{x} = 1$. It follows from the pronilpotency of $\mathfrak{g}$ and the above properties of $\bullet$ that the adjoint representation $y\mapsto \ad_y$ exponentiates to an action of $\exp(\mathfrak{g})$ on $\mathfrak{g}$ given by $e^y\mapsto e^{\ad_y}$.

Now given $\mathfrak{g}$ a differential graded Lie algebra, let $\xi\in\MC(\mathfrak{g})$ and $y\in\mathfrak{g}^0$. Define the \emph{gauge action} by
\[
e^y\cdot \xi = e^{\ad_y}\xi + (de^{\ad_y})y = \xi + \sum_{n=1}^{\infty} \frac{1}{n!}(\ad_y)^{n-1}(\ad_y\xi-dy).
\]
Then this indeed gives an action of $\exp(\mathfrak{g}^0)$ on $\MC(\mathfrak{g})$.

\begin{definition}
If $\mathfrak{g}$ is a pronilpotent differential graded Lie algebra, define the \emph{Maurer--Cartan moduli set} $\MCmoduli(\mathfrak{g})$ to be the quotient $\MC(\mathfrak{g})/\exp(\mathfrak{g}^0)$ of the set of Maurer--Cartan elements by the gauge action.
\end{definition}

Again it is clear that $\MCmoduli$ defines a functor on pronilpotent differential graded Lie algebras.

Given a (not necessarily pronilpotent) differential graded Lie algebra $\mathfrak{g}$ and an augmented formal commutative algebra $R$ with maximal ideal $R_+$, then the differential graded Lie algebra $\mathfrak{g}\otimes R_+$ is pronilpotent. We define the deformation functor $\Def_{\mathfrak{g}}$ associated to $\mathfrak{g}$ by $R\mapsto \MCmoduli(\mathfrak{g}\otimes R_+)$.

\subsection{Deformations of algebras}
We will now recall how the classical deformation functor associated to an $A_\infty$--algebra is equivalent to the deformation functor associated to a certain Lie algebra. We will also consider the analogous result for involutive $A_\infty$--algebras.

From now on, let $R$ be an augmented formal commutative algebra. We will need the notion of an $R$--linear $A_\infty$--algebra structure on a graded $R$--module. We shall assume all $R$--modules are free, so that they are of the form $R\otimes V$ for $V$ a graded vector space.

\begin{definition}
Let $V$ be a graded vector space. An \emph{$R$--linear $A_\infty$--structure} on $R\otimes V$ is an $R$--linear derivation $m\co R \otimes \widehat{T}_{\geq 1}\Sigma^{-1}V^*\rightarrow R \otimes \widehat{T}_{\geq 1}\Sigma^{-1}V^*$ of degree one such that $m^2=0$.
\end{definition}

As before, for an $R$--linear $A_\infty$--algebra $(V,m)$ the derivation $m$ can be seen to be determined completely by a collection of $R$--linear maps $\hat{m}_n\co (R\otimes V)\otimes_R  \dots \otimes_R (R\otimes V) \rightarrow R\otimes V$ of degrees $2-n$ satisfying the usual $A_\infty$--conditions. In particular if $\hat{m}_n = 0$ for all $n > 2$ then this is just a differential graded associative $R$--algebra on $R\otimes V$.

\begin{definition}
Let $(R\otimes V,m)$ and $(R\otimes W, m')$ be $R$--linear $A_\infty$--algebras. Then an \emph{$R$--linear $A_\infty$--morphism} is a map $\phi$ of $R$--linear associative algebras $\phi\co R\otimes \widehat{T}_{\geq 1}\Sigma^{-1}W^*\rightarrow R \otimes \widehat{T}_{\geq 1}\Sigma^{-1}V^*$ such that $m\circ \phi = \phi \circ m'$.
\end{definition}

Let $(R\otimes V,m)$ be an $R$--linear $A_\infty$--algebra. Since $m$ is $R$--linear it preserves the associative algebra ideal $R_+\otimes \widehat{T}_{\geq 1} \Sigma^{-1}V^*$ so any such $R$--linear $A_\infty$--structure on $R\otimes V$ determines an underlying $A_\infty$--structure on $V$ by taking the quotient with respect to this ideal. We denote this structure by $\reduction{m}\co\widehat{T}_{\geq 1}\Sigma^{-1}V^*\rightarrow \widehat{T}_{\geq 1}\Sigma^{-1}V^*$ and call it the \emph{reduction of $m$}. Similarly any $R$--linear associative algebra map $\phi\co R\otimes \widehat{T}_{\geq 1}\Sigma^{-1}W^*\rightarrow R \otimes \widehat{T}_{\geq 1}\Sigma^{-1}V^*$ induces a map $\reduction{\phi}\co \widehat{T}_{\geq 1}\Sigma^{-1}W^*\rightarrow \widehat{T}_{\geq 1}\Sigma^{-1}V^*$, the \emph{reduction of $\phi$}. 

If $V$ is a graded vector space with an involution then the $R$--algebra $R\otimes \widehat{T}\Sigma^{-1}V^*$ is an involutive $R$--algebra in the obvious way by extending $R$--linearly.

\begin{definition}
Let $V$ be a graded vector space with an involution. An \emph{$R$--linear involutive $A_\infty$--structure} on $R\otimes V$ is an $R$--linear derivation $m\co R \otimes \widehat{T}_{\geq 1}\Sigma^{-1}V^*\rightarrow R \otimes \widehat{T}_{\geq 1}\Sigma^{-1}V^*$ of degree one such that $m^2=0$ and $m$ preserves the involution: $m(x^*)=m(x)^*$.
\end{definition}

The additional requirement that $m$ preserves the involution translates to the same condition as before on each of the maps $\hat{m}_n$ with respect to the obvious involution on $R\otimes V$ obtained by extending $R$--linearly the involution on $V$.

\begin{remark}
Note that in our definition of an $R$--linear involutive $A_\infty$--structure we only consider $R$--linear involutions on $R\otimes V$ which are free in the sense that they arise as the natural extension of an involution on $V$. In terms of deformations this should be understood as deforming just the involutive $A_\infty$--structure, but not the involution itself.
\end{remark}

\begin{definition}
Let $(R\otimes V,m)$ and $(R\otimes W, m')$ be $R$--linear involutive $A_\infty$--algebras. Then an \emph{$R$--linear involutive $A_\infty$--morphism} is a map $\phi$ of $R$--linear associative algebras $\phi\co R\otimes \widehat{T}_{\geq 1}\Sigma^{-1}W^*\rightarrow R \otimes \widehat{T}_{\geq 1}\Sigma^{-1}V^*$ such that $m\circ \phi = \phi \circ m'$ and $\phi$ preserves the involution: $\phi(x^*)=\phi(x)^*$.
\end{definition}

Clearly if $(R\otimes V,m)$ in an $R$--linear involutive $A_\infty$--algebra then the reduction $(V,\reduction{m})$ is an involutive $A_\infty$--algebra. Similarly Let $\phi$ be an $R$--linear associative algebra map $\phi\co R\otimes \widehat{T}_{\geq 1}\Sigma^{-1}W^*\rightarrow R \otimes \widehat{T}_{\geq 1}\Sigma^{-1}V^*$ which preserves the involution. Then clearly the reduction $\reduction{\phi}$ is an associative algebra map $\phi\co \widehat{T}_{\geq 1}\Sigma^{-1}W^*\rightarrow \widehat{T}_{\geq 1}\Sigma^{-1}V^*$ which preserves the involution.

\begin{definition}
\Needspace*{3\baselineskip}\mbox{}
\begin{itemize}
\item If $(V,m)$ is an $A_\infty$--algebra then an \emph{$R$--deformation} of $(V,m)$ is an $R$--linear $A_\infty$--structure $m'$ on $R\otimes V$ such that $\reduction{m'}=m$. Two deformations are \emph{equivalent} if they are $A_\infty$--isomorphic as $R$--linear $A_\infty$--structures such that the reduction of this isomorphism is the identity.

Given an $A_\infty$--algebra $(V,m)$ then we define the classical deformation functor $\Def_{(V,m)}$ associated to $(V,m)$ which assigns to $R$ the set of equivalence classes of $R$--deformations of $(V,m)$.
\item If $(V,m)$ is an involutive $A_\infty$--algebra then an \emph{involutive $R$--deformation} of $(V,m)$ is an $R$--linear involutive $A_\infty$--structure $m'$ on $R\otimes V$ such that $\reduction{m'}=m$. Two deformations are \emph{equivalent} if they are $A_\infty$--isomorphic as $R$--linear involutive $A_\infty$--structures such that the reduction of this isomorphism is the identity.

Given an involutive $A_\infty$--algebra $(V,m)$ then we define the classical deformation functor $\Def_{(V,m)}$ associated to $(V,m)$ which assigns to $R$ the set of equivalence classes of $R$--deformations of $(V,m)$.
\end{itemize}
\end{definition}

\begin{remark}
There are of course other less strict ways, perhaps more natural and familiar to the reader, of defining deformations and equivalence of deformations. For example one could define an $R$--deformation to be an $R$--linear $A_\infty$--algebra $(R\otimes W,m')$ equipped with an isomorphism between $(W,\reduction{m'})$ and $(V,m)$ and two deformations are then equivalent if there is an $R$--linear $A_\infty$--isomorphism commuting with the isomorphisms of reductions in the obvious manner. However, it is easy to see that a deformation in this less strict sense is always equivalent to a deformation in our sense.
\end{remark}

The following well known theorem now makes precise the sense in which the Hochschild cohomology of an $A_\infty$--algebra `governs the deformation theory'. This point of view is well known, so a detailed proof will not be repeated here. The notes \cite{doubekmarklzima2007:deformationtheory} give a good introduction to this approach to deformations of algebras.

\begin{theorem}
Let $(V,m)$ be an $A_\infty$--algebra. Setting $\mathfrak{g}=\Sigma\hoch^\bullet(V,V)_{\geq 1}=\Der(\widehat{T}_{\geq 1}\Sigma^{-1}V^*)$ then $\Def_{\mathfrak{g}}\simeq \Def_{(V,m)}$.
\end{theorem}

\begin{proof}[Outline of proof]
Any $R$--linear derivation of $R\otimes \widehat{T}_{\geq 1}\Sigma^{-1}V^*$ is determined by its restriction to $\widehat{T}_{\geq 1}\Sigma^{-1}V^*$ so the Lie algebra of $R$--linear derivations of $R\otimes \widehat{T}_{\geq 1}\Sigma^{-1}V^*$ is isomorphic to $\Der(\widehat{T}_{\geq 1}\Sigma^{-1}V^*)\otimes R$. It is then straightforward to see that $\MC(\mathfrak{g}\otimes R_+)$ is in bijection with the set of $R$--deformations of $(V,m)$.

There is a group homomorphism $\exp(\mathfrak{g}^0\otimes R_+)\rightarrow \Aut_R(\widehat{T}_{\geq 1}\Sigma^{-1}V^*\otimes R)$ given by $e^y\mapsto \id + \sum_{n\geq 1} \frac{y^n}{n!}$ for $y\in\mathfrak{g}^0\otimes R_+$. The image of this map consists of those $R$--linear automorphisms whose reduction is the identity automorphism. Furthermore given $\phi$ in this image then $\log \phi = \sum_{n\geq 1} \frac{(-1)^{n+1}}{n}(\phi-\id)^n$ is a well-defined element of $\mathfrak{g}^0\otimes R_+$ and $\phi\mapsto e^{\log \phi}$ is an inverse on this image. It follows that there is an isomorphism of the gauge group $\exp(\mathfrak{g})$ of $\mathfrak{g}$ with the group of automorphisms whose reduction is the identity automorphism. A straightforward check shows that the two groups act on $\MC(\mathfrak{g}\otimes R_+)$ in the same way and the result then follows.
\end{proof}

It is completely clear from the proof of this theorem that an analogous result holds for involutive $A_\infty$--algebras. 

\begin{theorem}
Let $(V,m)$ be an involutive $A_\infty$--algebra. Setting $\mathfrak{g}=\Sigma\hoch^\bullet_+(V,V)_{\geq 1} = \Der_+(\widehat{T}_{\geq 1}\Sigma^{-1}V^*)$ then $\Def_{\mathfrak{g}} \simeq \Def_{(V,m)}$.
\end{theorem}

\begin{proof}
The proof is the same as above, together with the observation that degree $0$ derivations preserving the involution are precisely the derivations which exponentiate to an automorphism preserving the involution.
\end{proof}

\subsection{Deformations of cyclic algebras}
If $V$ is a graded vector space with a symmetric bilinear form of degree $d$ then $R\otimes V$ has a symmetric $R$--bilinear form $(R\otimes V)\otimes_R (R\otimes V) \rightarrow \Sigma^d R$. An $R$--linear derivation $m\co R \otimes \widehat{T}_{\geq 1}\Sigma^{-1}V^*\rightarrow R \otimes \widehat{T}_{\geq 1}\Sigma^{-1}V^*$ will be called an \emph{$R$--linear cyclic derivation} if the maps $\hat{m}_n\co (R\otimes V)\otimes_R \dots \otimes_R (R\otimes V)\rightarrow R\otimes V$ satisfy \autoref{eq:cycliccondition} with respect to this $R$--bilinear form.

\begin{definition}
Let $V$ be a graded vector space with a symmetric bilinear form of degree $d$. An \emph{$R$--linear cyclic $A_\infty$--structure} on $R\otimes V$ is an $R$--linear \emph{cyclic} derivation $m\co R \otimes \widehat{T}_{\geq 1}\Sigma^{-1}V^*\rightarrow R \otimes \widehat{T}_{\geq 1}\Sigma^{-1}V^*$ of degree one such that $m^2=0$.
\end{definition}

\begin{remark}
Note that our definition of a $R$--linear cyclic $A_\infty$--structure only considers $R$--bilinear forms on $R\otimes V$ which are free in the sense that they arise in the natural way from a bilinear form $V\otimes V\rightarrow \Sigma^{d}k$. In terms of deformations this should be understood as deforming just the cyclic $A_\infty$--structure, but not the bilinear form itself.
\end{remark}

\begin{definition}
Let $V$ and $W$ be graded vector spaces with degree $d$ symmetric bilinear forms. Let $(R\otimes V,m)$ and $(R\otimes W, m')$ be $R$--linear cyclic $A_\infty$--algebras. Then an \emph{$R$--linear cyclic $A_\infty$--morphism} is a map $\phi$ of $A_\infty$--algebras such that $\phi(\omega')=\omega$ where $\omega\in\Sigma^{-1}V^*\otimes \Sigma^{-1}V^*$ and $\omega'\in\Sigma^{-1}W^*\otimes \Sigma^{-1}W^*$ are the degree $d+2$ elements representing the bilinear forms.
\end{definition}

The condition that $\phi(\omega')=\omega$ can, of course, be restated explicitly as the same condition given in \autoref{rem:preservebilinear}, but understood with respect to the $R$--bilinear forms.

Clearly the reduction of an $R$--linear cyclic $A_\infty$--algebra or $R$--linear cyclic $A_\infty$--morphism is a cyclic $A_\infty$--algebra or cyclic $A_\infty$--morphism.

\begin{definition}
Let $V$ be a graded vector space with an involution and a symmetric bilinear form such that $\langle x^*, y^* \rangle = \langle x, y \rangle$. An \emph{$R$--linear cyclic involutive $A_\infty$--structure} on $R\otimes V$ is an $R$--linear cyclic derivation $m\co R \otimes \widehat{T}_{\geq 1}\Sigma^{-1}V^*\rightarrow R \otimes \widehat{T}_{\geq 1}\Sigma^{-1}V^*$ of degree one such that $m^2=0$ and $m$ preserves the involution: $m(x^*)=m(x)^*$.
\end{definition}

\begin{definition}
Let $V$ and $W$ be graded vector spaces with involutions and degree $d$ symmetric bilinear forms such that $\langle x^*, y^* \rangle = \langle x, y \rangle$. Let $(R\otimes V,m)$ and $(R\otimes W, m')$ be $R$--linear cyclic involutive $A_\infty$--algebras. Then an \emph{$R$--linear cyclic involutive $A_\infty$--morphism} is a map $\phi$ of $R$--linear involutive $A_\infty$--algebras which is also a map of $R$--linear cyclic $A_\infty$--algebras.
\end{definition}

\begin{definition}
\Needspace*{3\baselineskip}\mbox{}
\begin{itemize}
\item If $(V,m)$ is a cyclic $A_\infty$--algebra then a \emph{cyclic $R$--deformation} of $(V,m)$ is a cyclic $R$--linear $A_\infty$--structure $m'$ on $R\otimes V$ such that $\reduction{m'}=m$. Two deformations are \emph{equivalent} if they are $A_\infty$--isomorphic as $R$--linear cyclic $A_\infty$--structures such that the reduction of this isomorphism is the identity.

Given a cyclic $A_\infty$--algebra $(V,m)$ then we define the classical deformation functor $\Def_{(V,m)}$ associated to $(V,m)$ which assigns to $R$ the set of equivalence classes of cyclic $R$--deformations of $(V,m)$.
\item If $(V,m)$ is a cyclic involutive $A_\infty$--algebra then a \emph{cyclic involutive $R$--deformation} of $(V,m)$ is a cyclic involutive $R$--linear $A_\infty$--structure $m'$ on $R\otimes V$ such that $\reduction{m'}=m$. Two deformations are \emph{equivalent} if they are $A_\infty$--isomorphic as $R$--linear cyclic involutive $A_\infty$--structures such that the reduction of this isomorphism is the identity.

Given a cyclic involutive $A_\infty$--algebra $(V,m)$ then we define the classical deformation functor $\Def_{(V,m)}$ associated to $(V,m)$ which assigns to $R$ the set of equivalence classes of cyclic involutive $R$--deformations of $(V,m)$.
\end{itemize}
\end{definition}

The following well known theorem (see, for example, \cite{penkavaschwarz:ainfalgsmodspace}) makes precise the sense in which cyclic cohomology governs the deformation theory of cyclic $A_\infty$--algebras.

\begin{theorem}
Let $(V,m)$ be a cyclic $A_\infty$--algebra. Setting $\mathfrak{g}=\Der^{\cycl}(\widehat{T}_{\geq 1}\Sigma^{-1}V^*)$ then $\Def_{\mathfrak{g}}\simeq \Def_{(V,m)}$. If the bilinear form on $V$ is non-degenerate with degree $d$ then $\mathfrak{g}\cong\Sigma^{d+1}\CC^\bullet(V)_{\geq 1}$.
\end{theorem}

\begin{proof}
The proof is the same as before, together with the observation that degree $0$ cyclic derivations are precisely the derivations which exponentiate to an automorphism preserving $\omega\in\Sigma^{-1}V^*\otimes \Sigma^{-1}V^*$.
\end{proof}

We can now obtain the analogous result for cyclic involutive $A_\infty$--algebras immediately.

\begin{theorem}
Let $(V,m)$ be a cyclic involutive $A_\infty$--algebra. Setting $\mathfrak{g}=\Der^{\cycl}_+(\widehat{T}_{\geq 1}\Sigma^{-1}V^*)$ then $\Def_{\mathfrak{g}}\simeq \Def_{(V,m)}$. If the bilinear form on $V$ is non-degenerate with degree $d$ then $\mathfrak{g}\cong\Sigma^{d+1}\CD^\bullet_+(V)_{\geq 1}$.
\end{theorem}

\begin{proof}
The proof is an obvious combination of the previous proofs.
\end{proof}

\begin{example}
Consider the space of infinitesimal deformations, in other words deformations over $R=k[\epsilon]/(\epsilon^2)$. The moduli set of infinitesimal deformations of an involutive $A_{\infty}$--algebra $(V,m)$ is $\Hhoch^2_+(V,V)_{\geq 1}$. The moduli set of infinitesimal deformations of a cyclic involutive $A_\infty$--algebra $(V,m)$ with a non-degenerate bilinear form of degree $0$ is $\HCD_+^2(V)_{\geq 1}$.
\end{example}

\section{Generalisations to other operads}\label{sec:operads}
The purpose of this section is twofold. Firstly we will briefly outline how the ideas above can be generalised to other Koszul operads. Secondly we show how the above definition of an involutive $A_\infty$--algebra (or more generally an involutive homotopy $\mathcal{P}$--algebra) is a `correct' notion in the sense that it is equivalent to an algebra over a cofibrant replacement for the operad governing involutive associative algebras (or more generally for the operad governing involutive $\mathcal{P}$--algebras).

\begin{remark}
We will only deal with involutive homotopy $\mathcal{P}$--algebras parallel to \autoref{sec:noncyc}. When $\mathcal{P}$ is cyclic results parallel to \autoref{sec:cyc} concerning cyclic involutive homotopy $\mathcal{P}$--algebras and cyclic/dihedral cohomology should also generalise easily, for example by combining with the work in \cite{getzlerkapranov1995:cyclicoperads}. However the cyclic case is a little more complex and it's felt that such a generalisation here would take us too far afield.
\end{remark}

\subsection{Operad basics}
It is assumed that the reader is familiar with the theory of operads. However, for convenience we will briefly outline our notation and recall a couple of pertinent facts here, mainly from Ginzburg--Kapranov \cite{ginzburgkapranov1994:koszuloperads}.

Let $K$ be a semisimple associative algebra and denote by $\dgop(K)$ the category of differential graded operads $\mathcal{P}$ such that all the spaces $\mathcal{P}(n)$ are finite dimensional differential graded vector spaces and as algebras $\mathcal{P}(1)=K$. Given $\mathcal{P}\in\dgop(K)$ denote the cobar construction/dual operad by $\dual{\mathcal{P}}\in\dgop(K^{\mathrm{op}})$ so that there is a natural quasi-isomorphism $\dual\dual\mathcal{P}\rightarrow\mathcal{P}$, see \cite{ginzburgkapranov1994:koszuloperads}. If $\mathcal{P}\in\dgop(K)$ is a quadratic operad denote by $\mathcal{P}^!\in\dgop(K^{\mathrm{op}})$ its quadratic dual and call $\mathcal{P}$ Koszul if the natural map $\dual\mathcal{P}^!\rightarrow\mathcal{P}$ is a quasi-isomorphism.

In the case that $\mathcal{P}$ is Koszul, we take the following to be the definition of a homotopy $\mathcal{P}$--algebra.

\begin{definition}
Let $\mathcal{P}$ be a Koszul operad with Koszul dual $\mathcal{P}^!$. Then a \emph{homotopy $\mathcal{P}$--algebra} is an algebra over the operad $\dual\mathcal{P}^!$.
\end{definition}

Given $\mathcal{P}\in\dgop(K)$ then $P(n)$ is naturally a $(K, K^{\otimes n})$--bimodule. Let $V$ be a \emph{formal} left $K$--module, that is a left $K$--module which is an inverse limit of free graded left $K$--modules of finite rank. Then $V^{\otimes n}$ is a formal left $K^{\otimes n}$--module, where the tensor product is taken to be the \emph{completed} tensor product. Denote by $\cfree{\mathcal{P}}{V}$ the free formal $\mathcal{P}$--algebra generated by $V$, whose underlying space is given by
\[
\cfree{\mathcal{P}}{V} = \prod_{n=1}^{\infty} \mathcal{P}(n)\otimes_{K^{\otimes n}} V^{\otimes n}.
\]

\begin{definition}
Let $A$ be a $\mathcal{P}$--algebra. A \emph{graded derivation} of $A$ is a graded map $f\co A\rightarrow A$ such that for any $m \in \mathcal{P}(n)$
\[
f(m(a_1, \dots, a_n )) = \sum_{i=1}^n (-1)^{\epsilon} m(a_1, \dots, a_{i-1}, f(a_i), a_{i+1}, \dots, a_n)
\]
where $\epsilon = \degree{f}(\degree{m}+\degree{a_1}+\dots+\degree{a_{i-1}})$.The space spanned by all graded derivations forms a Lie subalgebra of the space $\IHom(A,A)$ of linear maps with the commutator bracket and is denoted by $\Der^\mathcal{P}(A)$.
\end{definition}

\begin{remark}
To avoid cluttered expressions, we will suppress the repetition of notation where possible and write $\Der(\cfree{\mathcal{P}}{W})$ for $\Der^{\mathcal{P}}(\cfree{\mathcal{P}}{W})$.
\end{remark}

If $V$ is a discrete graded left $K$--module, then $V^*$ is a formal graded left $K^{\mathrm{op}}$--module.

\begin{proposition}\label{prop:homotopyalgisderivation}
If $\mathcal{P}\in\dgop(K)$ is a Koszul operad with Koszul dual $\mathcal{P}^!$ then a homotopy $\mathcal{P}$--algebra structure on a (discrete) graded left $K$--module $V$ is equivalent to a Maurer--Cartan element in $\Der(\cfree{\mathcal{P}^!}{\Sigma^{-1}V^*})$.
\end{proposition}

\begin{proof}
This is entirely standard. See, for example, \cite[Proposition 4.2.14]{ginzburgkapranov1994:koszuloperads}.
\end{proof}

\subsection{Involutive \texorpdfstring{$\mathcal{P}$}{P}--algebras}
In \cite{braun:moduliklein} the functor $\mob$ taking $\mathcal{P}\in\dgop(k)$ to $\mob\mathcal{P}\in\dgop(k[\mathbb{Z}_2])$ was introduced.

The operad $\mob\mathcal{P}$ is obtained by freely adjoining an element $a$ to $\mathcal{P}(1)$ in degree $0$ and imposing the relations $da=0$, $a^2=1$ and the reflection relation $am=\tau_n(m)a^{\otimes n}$ for all $m\in\mathcal{P}(n)$. Here $\tau_n=(1\quad n)(2\quad n-1)(3\quad n-2)\ldots\in S_n$ is the permutation reversing $n$ labels.

\begin{definition}
Let $\mathcal{P}\in\dgop(k)$. Then an \emph{involutive $\mathcal{P}$--algebra} is an algebra over $\mob\mathcal{P}\in\dgop(k[\mathbb{Z}_2])$.
\end{definition}

A left $k[\mathbb{Z}_2]$--module is equivalent to a vector space with an involution. Therefore the operad $\mob\ass$ governs involutive associative algebras. The operad $\mob\com$ governs involutive commutative algebras, in other words commutative algebras with an involution $x\mapsto x^*$ such that $(xy)^*=(-1)^{\degree{x}\degree{y}}y^*x^* = x^*y^*$. Similarly, $\mob\lie$ governs involutive Lie algebras, in other words Lie algebras with an involution such that $[x,y]^* = (-1)^{\degree{x}\degree{y}}[y^*,x^*] = -[x^*,y^*]$.

\begin{remark}\label{rem:invskewinv}
One should compare the definition of an involutive Lie algebra with that of a skew-involutive Lie algebra in \autoref{def:skewinvlie}. Clearly any involutive algebra can be converted into a skew-involutive algebra and vice versa by replacing the involution $x\mapsto x^*$ with $x\mapsto -x^*$.
\end{remark}

\begin{proposition}\label{prop:freeinv}
Let $\mathcal{P}\in\dgop(k)$ and let $W$ be a formal left $k[\mathbb{Z}_2]$--module. As $\mathcal{P}$--algebras $\cfree{\mob\mathcal{P}}{W}\cong\cfree{\mathcal{P}}{W}$.
\end{proposition}

\begin{proof}
This follows from the observation that as $k$--linear vector spaces $\mob\mathcal{P}(n)\otimes_{k[\mathbb{Z}_2]^{\otimes n}} W^{\otimes n}\cong \mathcal{P}(n)\otimes W^{\otimes n}$.
\end{proof}

\begin{remark}
Of course it follows that $\cfree{\mathcal{P}}{W}$ has a natural involution making it into an $\mob\mathcal{P}$--algebra. Explicitly, the involution on $\cfree{\mathcal{P}}{W}$ is defined in the entirely obvious way by
\[
(m\otimes x_1 \otimes \dots \otimes x_n)^* = (-1)^\epsilon m \otimes x_n^*\otimes \dots \otimes x_1^*.
\]
\end{remark}

Let $\mathfrak{g}$ be a skew-involutive graded Lie algebra, so that $[\xi,\eta]^* = [\xi^*,\eta^*]$. Write $\mathfrak{g}_+$ for the $+1$ eigenspace, which is a Lie subalgebra. Write $\mathfrak{g}_-$ for the $-1$ eigenspace so that $\mathfrak{g}=\mathfrak{g}_+\oplus \mathfrak{g}_-$ as graded vector spaces and
\[
[\mathfrak{g}_-,\mathfrak{g}_-]\subset \mathfrak{g}_+.
\]
Furthermore, given $\xi\in\MC(\mathfrak{g}_+)$, equipping the graded vector space $\mathfrak{g}$ with the differential $d(\eta) = [\xi,\eta]$ then as \emph{differential} graded vector spaces $\mathfrak{g} = \mathfrak{g}_+ \oplus \mathfrak{g}_-$.

\begin{proposition}
Let $A$ be an $\mob\mathcal{P}$--algebra. Then by forgetting the involution $A$ is naturally a $\mathcal{P}$--algebra and $\Der^{\mathcal{P}}(A)$ is a skew-involutive Lie algebra, with involution $\xi\mapsto \xi^*$ defined by $\xi^*(x) = \xi(x^*)^*$.
\end{proposition}

\begin{proof}
First check $\xi^*$ is in $\Der^\mathcal{P}(A)$. Let $m\in\mathcal{P}(n)$ then
\begin{align*}
\xi^*(m(a_1, \dots, a_n )) 
&= (-1)^\epsilon\xi(\tau_n(m)(a_n^*,\dots, a_1^*))^*\\
&= \sum_{i=1}^n (-1)^{\epsilon'} \tau_n(m)(a_n^*, \dots, a_{i+1}^*, \xi(a_i^*), a_{i-1}^*, \dots, a_1^*)^*\\
&= \sum_{i=1}^n (-1)^{\epsilon''} m(a_1, \dots, a_{i-1}, \xi^*(a_i), a_{i+1}, \dots, a_n)
\end{align*}
where $\epsilon$, $\epsilon'$ and $\epsilon''$ arise from the Koszul sign rule, remembering that the involution preserves degree. Therefore $\xi^*$ is indeed a derivation and the result now follows from the observation that $\xi(\eta(x^*))^* = \xi^*(\eta^*(x))$.
\end{proof}

\begin{remark}\label{rem:mobderpreservesinv}
In this case we denote the decomposition of $\Der^\mathcal{P}(A)$ as $\Der^{\mathcal{P}}(A)=\Der_+^\mathcal{P}(A)\oplus \Der_-^\mathcal{P}(A)$. Then of course $\Der_+^\mathcal{P}(A)$ are those derivations which preserve the involution and $\Der^{\mob\mathcal{P}}(A)=\Der^{\mathcal{P}}_+(A)$.
\end{remark}

\subsection{Involutive homotopy \texorpdfstring{$\mathcal{P}$}{P}--algebras}
If $V$ is a graded vector space with an involution then $V^*$ is naturally a formal left $k[\mathbb{Z}_2]^{\mathrm{op}}$--module, where the involution is defined by $\phi^*(v)=\phi(v^*)$. Additionally $V^*$ can also be seen as a formal left $k[\mathbb{Z}_2]$--module where the involution is defined by $\phi^*(v)=-\phi(v^*)$. Of course $k[\mathbb{Z}_2]$ is commutative so both structures could be regarded as left $k[\mathbb{Z}_2]$--module structures but it is important to distinguish the two.

\begin{definition}
Let $\mathcal{P}\in\dgop(k)$ be Koszul with Koszul dual $\mathcal{P}^!\in\dgop(k)$ and let $V$ be a vector space with an involution. Then an \emph{involutive homotopy $\mathcal{P}$--algebra structure} on $V$ is a Maurer--Cartan element in $\Der_+(\cfree{\mathcal{P}^!}{\Sigma^{-1}V^*})$.
\end{definition}

The following theorem asserts that an involutive homotopy $\mathcal{P}$--algebra in this sense is in fact precisely a homotopy $\mob\mathcal{P}$--algebra.

\begin{theorem}
Let $\mathcal{P}\in\dgop(k)$ be Koszul with Koszul dual $\mathcal{P}^!$ and let $V$ be a graded vector space with an involution. Then $\mob\mathcal{P}$ is also Koszul and there is an isomorphism of Lie algebras $\Der(\cfree{(\mob\mathcal{P})^!}{\Sigma^{-1}V^*}) \cong \Der_+(\cfree{\mathcal{P}^!}{\Sigma^{-1}V^*})$.
\end{theorem}

\begin{proof}
By \cite[Theorem 3.33]{braun:moduliklein} $\mob\mathcal{P}$ is also Koszul and there is an isomorphism $(\mob\mathcal{P})^!\cong\mob(\mathcal{P}^!)$. Because $(\mob\mathcal{P})^!(1) = k[\mathbb{Z}_2]^{\mathrm{op}}$ and $\mob(\mathcal{P}^!)(1)=k[\mathbb{Z}_2]$ and the operad isomorphism between them takes the non-trivial element $a\in \mathbb{Z}_2$ to $-a$, then $\Der(\cfree{(\mob\mathcal{P})^!}{\Sigma^{-1}V^*}) \cong \Der(\cfree{\mob(\mathcal{P})^!}{\Sigma^{-1}V^*})$. The right hand side is then isomorphic to $\Der_+(\cfree{\mathcal{P}^!}{\Sigma^{-1}V^*})$ by \autoref{prop:freeinv} and \autoref{rem:mobderpreservesinv} as required.
\end{proof}

\begin{corollary}
An involutive homotopy $\mathcal{P}$--structure on $V$ is the same as a homotopy $\mob\mathcal{P}$--algebra structure on $V$.
\qed
\end{corollary}

For a homotopy $\mathcal{P}$--algebra $V$ represented by $m\in\MC(\Der(\cfree{\mathcal{P}^!}{\Sigma^{-1}V^*}))$ the complex $C^\bullet_\mathcal{P}(V) = \Sigma^{-1}\Der(\cfree{\mathcal{P}^!}{\Sigma^{-1}V^*})$ with differential $d(\xi)=[m,\xi]$ is the $\mathcal{P}$--cohomology complex of $V$. Clearly if this is an involutive homotopy $\mathcal{P}$--algebra then this decomposes into the involutive $\mathcal{P}$--cohomology complex (which is the $\mob\mathcal{P}$--cohomology complex) and the skew-involutive $\mathcal{P}$--cohomology complex.

\begin{example}
Let $\mathfrak{g}$ be an involutive differential graded Lie algebra (or more generally an involutive $L_\infty$--algebra). Then the Chevalley--Eilenberg cohomology decomposes as $\HCE^\bullet(\mathfrak{g},\mathfrak{g})\cong\HCE^\bullet_+(\mathfrak{g},\mathfrak{g})\oplus\HCE^\bullet_-(\mathfrak{g},\mathfrak{g})$. If $\mathfrak{g}$ is equipped with the involution $v^*=-v$ then $\HCE^\bullet_-(\mathfrak{g},\mathfrak{g})=0$.

 Dually if $C$ is an involutive differential graded commutative algebra (or more generally an involutive $C_\infty$--algebra) then the Harrison cohomology decomposes as $\Hharr^\bullet(C,C) \cong \Hharr^\bullet_+(C,C)\oplus \Hharr^\bullet_-(C,C)$. If $C$ is equipped with the identity involution then $\Hharr^\bullet_-(C,C)=0$.
\end{example}

\bibliography{references}
\bibliographystyle{alphaurl}

\end{document}